\documentclass[11pt,letterpaper]{amsart}
 \usepackage[T1]{fontenc}
 \usepackage[cp1250]{inputenc}
\usepackage{xcolor}
\usepackage[normalem]{ulem}

 \newtheorem{theorem}{Theorem}

\theoremstyle{definition}

\theoremstyle{remark}
\newtheorem{remark}{Remark}
\theoremstyle{lemma}
\newtheorem{lemma}{Lemma}

\begin{document}

\title{Analogues of $s$-potential and $s$-energy \\ of mass distribution \\ on Cantor dyadic group \\ and their relation \\ to Hausdorff dimension}


\author[SKVORTSOV]{VALENTIN SKVORTSOV}
\address{Lomonosov Moscow State University, Mathematics Department 
and Moscow Center of Fundamental and Applied Mathematics\\
Moscow, 119991 Russia}
\email{vaskvor2000@yahoo.com}

\maketitle
\thispagestyle{empty}








\begin{abstract}
 We introduce an analogue of Riesz $s$-potetial and $s$-energy, $0<s<1$, of a mass distribution $\mu$ on the Cantor dyadic group $G$ by defining a respactive $s$-kernel. Then we relate Hausdorff dimension of a set $E\subset G$ to the value of $s$-energy of the mass distribution $\mu$ on this set $E$. Namely we prove that if on a set $E$ there exists a mass distribution $\mu$  with finite $s$-energy, then the Hausdorff dimention of $E$ is at least $s$.
    The same condiion can be expressed also in terms of Fourier coefficients of $\mu$ with respect to Walsh system on the group $G$. 


\end{abstract}

\section{Introduction}\label{s1}
The classical theory of potential and energy and its connection with harmonic analysis on a one-dimensional torus $T$ and in $\mathbb R$ has been widely covered in the literature (see \cite{Kahane}, \cite{Falconer}).

Riesz $s$-potential on $\mathbb R$ or on $T$ is defined as a convolution of a {\it kernel of order} $s,$ $0<s\leq 1,$ $|x|^{-s}$
($|\sin(t/2)|^{-s}$  in the case of $T$) and a positive finite measure ({\it mass distribution}) $\mu$ supported on a set. $s$-Energy is the integral of $s$-potential with respect to $\mu$.

One of the topics which links  potential and energy theory to Fourier analysis is some deep results which relate Hausdorff dimension of a set to potential of a mass distribution on this set  (see \cite[Ch. III]{Kahane} or
\cite[Theorem 4.13]{Falconer}).

A problem of extending the potential theory and its methods to a case of groups different from one-dimentional torus was touched upon only in a few publication (see for example \cite{var} and \cite{vod} for the use of tools from potential theory to study properties of connected Lie groups).

In the present paper we introduce in section 3 an analogue of  $s$-kernal for the Cantor dyadic group and study its properties. Having this kernel we define in Section 4  $s$-potetial and $s$-energy on the Cantor group and obtain in Subsection 4.1 the main result of the paper on relation of Hausdorff dimension of a set $E\subset G$ to the value of $s$-energy of a mass distribution $\mu$ on this set $E$. In Subsection 4.2 the same relation is expressed also in terms of Fourier coefficients of $\mu$ with respect to Walsh system (see \cite{gol}) on the group $G$. 

\section{Preliminaries}
\noindent 
Cantor dyadic group is the set of sequences whose entries are either $0$ or $1$, namely,
sequences of the form $x=\{x_j\}_{j=1}^{\infty}$, where the group operation is a coordinate-wise addition modulo 2, i.e.,
the sum of two sequences $x=\{x_j\}_{j=1}^{\infty}$ and $y=\{y_j\}_{j=1}^{\infty}$
is the sequence $z=\{z_j\}_{j=1}^{\infty}$ given by
$$z_j=\begin{cases}
	0&\text{if $x_j=y_j$,}\\
	1&\text{if $x_j\neq y_j$.}
\end{cases}$$

The group G contains a strictly
decreasing sequence
\begin{equation*} G=G_{0}\supset G_{1} \supset G_{2}\ldots\supset G_{n} \supset \ldots
\end{equation*} 
of  subgroups $G_n:=\{x\in G: x_j=0, 1\leq j \leq n\}$ forming a base of neighborhoods of the zero element.
The subgroups $G_n$ and all their cosets are clopen sets in the topology introduced above. Let $K_n$
denote an arbitrary coset of the subgroup $G_n$, and let $K_n(x)$ be the coset of $G_n$ containing the element
$x$, so that $K_n(x)=G_n+x$. It is clear that the order of the quotient group $G/G_n$ is $2^n$. When it is needed to list
all the cosets $K_n$, we number them as $K^i_n,$
$1\leq i \leq 2^n,$  so that $K^1_n= G_n$ and $K^1_n= K_{n+1}^1\cup K_{n+1}^2$ . Then

\begin{equation}\label{sum}
	\{0\}=\bigcap_{n=1}^{\infty}K^1_n \, \,  \,  \text{and} \, \, \,
	G\setminus \{0\}=\bigcup_{n=1}^{\infty}K^2_n.
\end{equation}
Let $\lambda$ denote a normalized Haar measure on $G$. Since the measure $\lambda$ is translation invariant,
it follows that, for all 
$n\geq 0$
$$\lambda(G_n)=\lambda(K_n)=\frac{1}{2^n}.$$
It is known (see, for example, \cite{Pon}) that $G$ as a topological group is metrizable and a metric can be given as a distance between elements $x, \, y\, \in G$ defined by
\begin{equation}\label{rho}
	\rho(x, y)=\sum_{i=1}^{\infty}2^{-i}|x_i-y_i|.
\end{equation}
Note that with this metric
\begin{equation}\label{rhoKn}
	\rho(x, y)<2^{-n }\, \, \text{if and only if} \, \, x, \, y \in K_n(x)=K_n(y).
\end{equation}

Letting $n$ be given  by its dyadic expansion 
\begin{equation}
	n=\sum_{i=0}^kn_i2^i,\label{drops}
\end{equation}
where $n_i\in\{0,1\}$ if $i<k$ and~$n_k=1,$
the $n$-th Walsh function on this group is defined as 
\begin{equation}\label{wnG}
	w_n(x)=(-1)^{\sum_{i=0}^kn_ix_{i+1}}.
\end{equation}
Note that the system of Walsh functions constitutes the system of characters of the group $G$. 

The following properties of the Walsh functions will be used in our further computation. Let $2^{n-1}\leq k<2^n.$ Then
\begin{equation}\label{1-1}
	w_k(x)=\begin{cases}
		1&\text{if} \, \, \, x\in K_n^1,\\
		-1&\text{if} \, \, \, x\in K_n^2.
	\end{cases}
\end{equation}
and
\begin{equation}\label{0}
	\int_{K_{n-1}}w_k\, d\lambda=0,
\end{equation}

\section{Kernel of $s$-potential on Cantor group}

We define {\it $s$-kernel,} $0<s\leq 1,$ by 

$$\phi_s(x)=\begin{cases}
	0&\text{if $x=0$,}\\
	2^{s(n-1)}&\text{if $x\in K_n^2, \, \, n\geq 1$.}
\end{cases}$$
Note that the kernel  $\phi_s$ is positive everywhere on $G\setminus\{0\}$ and
\begin{equation}\label{>}
	\phi_s(x)\geq 2^{ns}
	\, \, \text{if} \, \, x\in  K_n^1=G_n.\end{equation}

Consider trancated kernels:
$$\phi_s^n(x)=\begin{cases}
	\phi(x)& \text{if $x\in G\setminus G_n \, \, n\geq 1$} \\
	2^{ns}&\text{if $x\in G_n\setminus\{0\},$.} \\
	0&\text{if $x=0$.}
\end{cases}$$
It is clear that
$\phi_s^n\nearrow \phi_s$ when $n\to \infty.$ Note that the inequality (\ref{>}) implies an estimate
\begin{equation}\label{psin>}
	\phi_s^n(x)\geq 2^{ms}
	\, \, \text{if} \, \, x\in  K_m^1=G_m \, \, \text{and} \, \,  m\leq n.
\end{equation}

We need the following properties of the kernel $\phi_s^{n}$ and its  Walsh-Fourier coefficients.
\begin{lemma}\label{lemma}
	The kernel $\phi_s^n$ is a polinomial with respect to the Walsh system with positive coefficient $\widehat{\phi_s^{n}}(k)$ for $1\leq k<2^n$ and 
	$\widehat{\phi_s^n}(k) \nearrow \widehat{\phi_s}(k)$ for each $k$ when $n\to \infty.$
\end{lemma}

\begin{proof}
	It is easy to check, using (\ref{0}), that $\widehat{\phi_s^{n}}(k)=0,$ for $k\geq 2^n.$
	If $2^{n-1}\leq k<2^n,$ then due to (\ref{1-1}), (\ref{0}) and the definition of $\phi_s^{(n)}$ we have  $$\widehat{\phi_s^{n}}(k)=\int_{K_n^1}\phi_s^n d\lambda - \int_{K_n^2}\phi_s^n d\lambda= 2^{ns}2^{-n}-2^{(n-1)s}2^{-n}=2^{n(s-1)}(1-2^{-s})>0.$$ A similar computation, based on using (\ref{psin>}), shows that $\widehat{\phi_s^{n}}(k)>0$ also for $2^{m-1}\leq k<2^m$ with $m<n.$ The inequality$\int_{K_m^1} \phi^{n+1}_s> \int_{K_m^1} \phi^n_s$ for $m<n$ implies that 
	$\widehat{\phi_s^{n+1}}(k)>\widehat{\phi_s^n}(k)$ if $2^{m-1}\leq k<2^m$ with $m\leq n.$ Moreover, $\widehat{\phi_s^{n+1}}(k)>0$ for $2^{n}\leq k<2^{n+1}$ while  
	$\widehat{\phi_s^n}(k)=0$ for the same $k.$ This proves that the sequence $\{\widehat{\phi_s^n}(k)\}_n$ increases for each fixed $k$, The limit of this sequence is $\widehat{\phi_s}(k)$ because passage to the limit under integral sign in integrals $\int_G\phi_s^n(x) w_k(x)d\lambda$ is justified by the inequality $|\phi_s^n w_k|<\phi_s$ with $\phi_s$ being obviously summable on $G$ with respect to the measure $\lambda$
\end{proof}

\section{$s$-potentia,  $s$-energy and Hausdorff dimension}

Having $s$-kernel, we define, in a way similar to the case of the real line or circle $T$ (see \cite{Kahane} and \cite{Falconer}), $s$-potential, with $0<s< 1,$ at a pont $x \in G$  as a convolution of $s$-kernel and 
a mass distribution $\mu$ 
$$\Phi_s(x)=\int_G\phi_s(x-y)\,d\mu(y)$$
and $s$-energy of $\mu$ as
$$I_s(\mu)=\int_G \Phi_s(x)\, d\mu(x).$$
In the same way we define $s$-potential $\Phi_s^n$ and $s$-energy $I_s^n$ corresponding to the kernel $\phi_s^n.$

We recall (see \cite{Falconer}, \cite{Rod}) the definition of Hausdorff measure and Hausdorff dimension of a set $A$ in a metric space $M$. For any $\delta$  we define  
\begin{equation}\label{Hsdelta}
	H_{\delta}^s(A)=\inf\left\{\sum_{i=1}^{\infty}(\text{diam} \, I_i)^s :  \, A\subset \cup_{i=1}^{\infty} I_i, \, \text{diam}\,  I_i < \delta \, \, \text{for all} \, \,  n  \right\}.
\end{equation}
A cover $\{I_i\}_i$ of a set $A$ in the above definition is called $\delta$-cover. 
As $\delta$ decreases, the infimum in (\ref{Hsdelta}) increases and we define the {\it$s$-dimensional Hausdorff measure} of $A$ as 
\begin{equation}\label{Hsmeasure}
	H^s(A)=\lim_{\delta\to 0}H_{\delta}^s(A).
\end{equation}
The {\it Hausdorff dimension} of a set $A$ is defined as
\begin{equation}\label{Hdim}
	\dim_{H}(A)=\inf\{s: H^s(A)=0\} \, (=\sup\{s: H^s(A)=\infty\}).
\end{equation}
$H^s$ is known (see \cite{Rod}) to be a metric outer meassure, and so it is a measure on the class of Borel sets.

\subsection{Relation of Hausdorff dimension of a set to the value of $s$-energy of a mass distribution on this set}

Here and everywhere in what follows we assume, that $0<s<1.$
\begin{theorem}\label{main}
	Let $E$ be a subset of the group $G$. If there is a non-atomic mass distribution $\mu$ on the set $E$ which has finite $s$-energy, then $E$ has Hausdorff dimention  at least $s$.
\end{theorem}
\begin{proof}.
Consider
$$E_1=\left\{x\in E: \, \, \overline{\lim_{n\to\infty}}\mu(K_n(x))2^{ns}>0\right\}.$$
Having fixed $x\in E_1$, we can find for any $\varepsilon>0$ a sequence $\{n_i\}$ so that $$\mu(K_{n_i}(x))\geq \varepsilon2^{-n_is}.$$
Due to continuity of measure $\mu,$ we can choose for each $i$ a number $m_i>n_i$ such that 
\begin{equation}\label{geq}
	\mu(K_{n_i}(x)\setminus K_{m_i}(x))\geq 2^{-1} \varepsilon2^{-n_is}.
\end{equation}
We can arrange that $n_{i+1}>m_i$ so that $K_{n_i}(x))\setminus K_{m_i}(x)$ with different $i$ do not intersect. Then having in mind (\ref{geq}), (\ref{>}) and the fact that if $y\in K_n(x)$ then $x-y\in K_n^1$, we estimate the $s$-potential at $x$:
\begin{gather*} \Phi_s(x)=\int_G\phi_s(x-y)\,d\mu(y)\geq\sum_i^{\infty}\int_{K_{n_i}(x))\setminus K_{m_i}(x)}\phi_s(x-y)\,d\mu(y)\\ \geq \sum_i^{\infty}2^{-1} \varepsilon2^{-n_is}2^{n_is}=\infty.
\end{gather*}
As   $\int_G\Phi_s(x)\,d\mu(x)<\infty$ by assumption, we get $\mu(E_1)=0$.

Denote $ F=E\setminus E_1.$ 
We have $$\lim_{n\to\infty}\mu(K_n(x))2^{ns}=0 \,\, \, \text{if} \, \, \, x\in F$$
Fix any $c>0$ which can be taken as small as we wish. Then for each point of $F$ we can choose a number $n(x)$ such that $\mu(K_n(x))<c2^{-ns}$ for all $n\geq n(x).$
Consider a set $F_{n_0}=\{x\in F: n(x)=n_0\}.$ Note that if $x\in F_{n_0}$ then 
\begin{equation}\label{Knx} \mu(K_n(x))<c2^{-ns} \, \, \text{for all}\, \,  n\geq n_{n_0}.
\end{equation}

Let $\{I_i\}_i$ be a $\delta$-cover of the set $F$ with 
$\delta=2^{-n_0}.$  Let $I_{i(x)}$ be the element of this cover for which $x\in F_{n_0}\cap I_{i(x)}$ and assume that $2^{-(n+1)}\leq \text{diam}\, I_{i(x)}<2^{-n},$ $n\geq n_0.$ Then having in mind (\ref{rhoKn})  we obtain that $I_{i(x)}\subset K_n(x).$  Then, using  (\ref{Knx}), we obtain
$$\mu I_{i(x)}\leq\mu K_n(x)< c2^{-ns}= c 2^s2^{-(n+1)s}\leq c2^s(\text{diam} \, I_{i(x)})^s$$
so that
$$\mu F_{n_0}\leq\sum\{\mu I_i : I_i \, \text{intersects} \,  F_{n_0}\}\leq  c2^s\sum_i(\text{diam} \, I_i)^s$$ and according to (\ref{Hsdelta}) and (\ref{Hsmeasure}) $\mu F_{n_0}\leq c2^SH_{\delta}^s(F)\leq c2^sH^s(F).$ Note that the $F_{n_0}$ increases to $F$ with increasing $n_0.$  So it follows from the continuity of measure $\mu$  that $\mu F\leq c2^sH^s(F)$ and we finally obtain $$H^s(E)\geq H^s(F)\geq 2^{-s}\mu F/c=2^{-s}\mu E/c.$$ As $c$ can be arbitrary small and $\mu E$ is supposed to be positive, we get $H^s(E)=\infty.$  Hence by (\ref{Hdim}) $\dim_H(E)\geq s.$ and the theorem is proved.
\end{proof}

\subsection{$s$-energy in terms of Walsh-Fourier coefficients}

\begin{theorem}\label{Fourier}
	The $s$-energy of a mass distribution $\mu$ on $G$ can be expressed as
	$$I_s(\mu)=\sum_{k=0}^{\infty}\widehat{\phi_s}(k)(\widehat{\mu}(k))^2$$
$($infinite values are not excluded in this equation$)$.
\end{theorem}
\begin{proof}
	The definitions of the kernels $\phi_s$ and $\phi_s^n$ imply
	$$\Phi_s^n(x)=\int_G\phi_s^n(x-y)\,d\mu(y) \nearrow \int_G\phi_s(x-y)\,d\mu(y)=\Phi_s(x) \, \, \text{with}\, \,  n\to \infty.$$
	Then 
	\begin{equation}\label{222}
		I_s^n(\mu)=\int_G \Phi_s^n(x)\, d\mu(x)\nearrow \int_G \Phi_s(x)\, d\mu(x)=I_s(\mu).
	\end{equation}
	The last pasage to the limit does not exclude the infinite value (remember that all elements of the sequences in the above equations are positive). 
	Using the known formula (see \cite{Rudin}) for Fourier transfprm of convolution on groups, we get .
	$$\widehat{\Phi_s^n}(k)=\widehat{\phi_s^n}(k)\widehat{\mu}(k).$$
	As $\phi_s^n$ is Walsh polinomial (see Lemma \ref{lemma}) we have
	$$\Phi_s^n(x)=\sum_{k=0}^{2^n-1}\widehat{\phi_s^n}(k)\widehat{\mu}(k)w_k(x)$$
	Integrating over $G$ with respect to $\mu$ we obtain 
	$$I_s^n=\sum_{k=0}^{2^n-1}\widehat{\phi_s^n}(k)(\widehat{\mu}(k))^2$$
	Now using (\ref{222}) and Lemma \ref{lemma} we can pass to the limit in the last equality and obtain the statement of the theorem.
\end{proof}
As a corollary of Theorem \ref{main} and \ref{Fourier} we get the following result which relates Hausdorff dimention of a set on $G$ to Walsh-Fourier coefficients of a mass distribution supported on this set. 
\begin{theorem}
	Let $E$ be a subset of the group $G$. If there is a non-atomic mass distribution $\mu$ on the set $E$ such that the series $\sum_{k=0}^{\infty}\widehat{\phi_s}(k)(\widehat{\mu}(k))^2$ is convergent, then $E$ has Hausdorff dimention  at least $s$.
\end{theorem}
\begin{remark}
	Elementary but rather tedious computation shows that Walsh-Fourier coefficient of the kernel $\phi_s$ satisfies the condition $\widehat{\phi_s(k)}=O(k^{s-1})$. So convergence of the series $\sum k^{s-1}(\widehat{\mu}(k))^2$ is a sufficient condition for the finiteness of the $s$-energy and the corresponding inequality for the Hausdorff dimension of the respective set.
\end{remark}

\end{document}